\documentclass[11pt]{amsart}

\usepackage{amssymb}
\usepackage{verbatim}

\newtheorem{Th}{Theorem}[section] \newtheorem{Cor}[Th]{Corollary}
\newtheorem{Lem}[Th]{Lemma} \newtheorem{Prop}[Th]{Proposition}

\numberwithin{equation}{section}
\renewcommand{\theequation}{\thesection.\arabic{equation}}

\def\aut#1{\operatorname{Aut}(#1)}
\def\inn#1{\operatorname{Inn}(#1)}
\def\ia#1{\operatorname{IA}(#1)}
\def\iat{\operatorname{IA}}
\def\gl#1{\operatorname{GL}(#1)}
\def\End#1{\operatorname{End}(#1)}

\def\iat{\operatorname{IA}}
\def\id{\operatorname{id}}
\def\inv{{}^{-1}}
\def\rank{\operatorname{rank}}
\def\str#1{\langle#1\rangle}

\def\f{\varphi}

\def\s{\sigma}

\def\cB{{\mathcal B}}

\def\cD{{\mathcal D}}

\def\N{\mathbf N}
\def\Z{\mathbf Z}

\renewcommand{\ge}{\geqslant}

\begin{document}

\setcounter{page}{113}

\copyrightinfo{2002}{American Mathematical Society}

\begin{abstract}
Considering a particular case of a problem
posed by S.~Shelah, we prove that the automorphism
group of an infinitely generated free nilpotent
group of cardinality $\lambda$ first-order interprets
the full second-order theory of the set $\lambda$
in empty language.
\end{abstract}

\title[Automorphism groups of free nilpotent groups]{On the logical
strength of the automorphism groups of free nilpotent groups}
\author{Vladimir Tolstykh}
\address{Department of Mathematics \\ Istanbul Bilgi University\\
Ku\c{s}tepe 80310 \c{S}i\c{s}li-Istanbul \\ Turkey}
\email{vladimirt\symbol{64}bilgi.edu.tr}
\maketitle

\section{Introduction}

In his paper \cite{ShCler} of 1976, S.~Shelah suggested
a general program of the study of the logical
strength of first-order theories of the
automorphism groups of free algebras.
Recently he has again attracted attention
to that program in his survey \cite{ShMathJap}.
Namely, the problem 3.14 from \cite{ShMathJap}
asks for which varieties $\text{\bf V}$ of
algebras, letting $F_\lambda$ for a free
algebra in $\text{\bf V}$ with $\lambda \ge \aleph_0$
free generators, can we syntactically
interpret in the first-order theory
of $\aut{F_\lambda}$ the full second-order
theory of the set $\lambda$ in empty
language (possibly
for sufficiently large cardinals $\lambda$).
Recall that a theory $T_0$ in a logic
$\mathcal L_0$ is said to be {\it syntactically
interpretable} in a theory $T_1$ in
a logic $\mathcal L_1$ if there is
a mapping $\chi \to \chi^*$ from
the set of all $\mathcal L_0$-sentences
to the set of all $\mathcal L_1$-sentences
such that
$$
\chi \in T_0 \iff \chi^* \in T_1.
$$

It should be pointed out that one the main results of
\cite{ShCler} states that for {\it any} variety of
algebras $\text{\bf V}$ the first-order theory of the
{\it endomorphism semi-group} $\End{F_\lambda}$
syntactically interprets
$\operatorname{Th}_2(\lambda)$ provided that a
cardinal $\lambda$ is greater than or
equal to the power of the language of $\text{\bf V}.$
The situation with the automorphism groups
seems to be more difficult and the reason is
obvious: despite being complicated
in many cases, the endomorphism semi-groups of
free algebras, as Shelah's analysis in \cite{ShCler}
demonstrates, can be viewed as combinatorial
objects.

There exist only a few examples of varieties for
which the Shelah's problem is completely
investigated:
some varieties have the desired property
(for instance, the variety of all vector spaces
over an arbitrary division ring and
the variety of all groups \cite{ToAPAL,ToJLM2}), some do not
(the variety of all sets with no structure
\cite{Sh1,Sh2}, the automorphism groups
of free algebras are here {\it symmetric
groups}). To the best of the author's knowledge, there
are no general results on the subject (however,
Shelah introduces in \cite[\S \ 3]{ShMathJap} a wide class
of so-called Aut-decomposable varieties, which
are in many ways analogous to the variety of all sets).

The purpose of the present paper is to prove that the
automorphism groups $\aut{F_\lambda}$ of
free groups $F_\lambda$ in all varieties of nilpotent
groups $\mathfrak N_s$ with $s \ge 2$ are logically
strong enough to interpret by means of first-order
logic the full second-order theory of $\lambda$
for all infinite $\lambda.$
We also consider a number of related questions;
it is proved, in particular, that the first-order
theory of the automorphism group of a finitely
generated free nilpotent group of class $\ge 2$
is unstable and undecidable. The author is
grateful to Oleg Belegradek, Edward Formanek
and Alexandre Iwanow for helpful discussions.

Some of the results of this paper were announced at
the International Conference ``Logic and Algebra''
(Istanbul, 2001); the author would like to express his
gratitude to the organizers of the Conference for
their warm hospitality.

\section{Reducing nilpotency class}

Suppose that $N$ is a free nilpotent group of class
$s \ge 2$ and let $K_m(N),$ where $m$ is
a naturalnumber, denote the kernel of the
homomorphism from the group $\aut N$ to the group
$\aut{N/N_{m+1}}$ induced by the natural homomorphism
$N \to N/N_{m+1},$ from $N$ to the free nilpotent
group $N/N_{m+1}$ of nilpotency class $m.$
In particular, $K_1(N)$ is equal to $\ia N,$
to the subgroup of so-called IA-{\it automorphisms} of $N,$
and $K_s(N) =\{\id\}.$

\begin{Lem} \label{CommModK+1}
Suppose that $\gamma$ is an $\iat$-automorphism.
Then $\gamma$ commutes with every element
of the subgroup $K_m(N)$ modulo the subgroup
$K_{m+1}(N).$
\end{Lem}

\begin{proof}
According to \cite{Andr}, the groups $K_m(N)$ form the
lower central series of the group $K_1(N)=\ia N;$ every element
of an arbitrary group $G$ commutes with the elements of
the $k$th term of the lower central series
of $G$ modulo the $(k+1)$th term \cite[Section 5.3]{MKS}.
\end{proof}

Like in our previous papers \cite{ToJLM1,ToCamb}, any
automorphism $\theta$ of $N,$ which inverts all
elements of some basis of $N$ will be called a {\it
symmetry}.

\begin{Lem} \label{SymmBasics}
Let $\theta$ be a symmetry.

\mbox{\em (a)} Suppose that $c$ is an element
of $N_m.$ Then $\theta$ either fixes $c$ modulo
$N_{m+1}$ {\em(}when $m$ is even{\em)}, or inverts
$c$ modulo $N_{m+1}$ {\em(}when $m$ is odd{\em)};

\mbox{\em (b)} Suppose that $\gamma$ is an
element of $K_m(N).$ Then the conjugate
of $\gamma$ by $\theta$ either equals
to $\gamma$ modulo $K_{m+1}(N)$ {\em(}when
$m$ is even{\em)} or to the inverse of $\gamma$
modulo $K_{m+1}(N)$  {\em(}when
$m$ is odd{\em)}.
\end{Lem}

\begin{proof}

(a) Assume that $\cB$ is a basis of $N$ such
that $\theta$ sends each element of $\cB$ to
its inverse. Since the group $N_m/N_{m+1}$ is abelian it
suffices to prove that $\theta$ acts in a prescribed way on
generators $[x_{i_1},x_{i_2},\ldots,x_{i_m}] N_{m+1},$ where
$x_{i_1},\ldots,x_{i_m}$ are elements of $\cB$
\cite[Section 5.3]{MKS}. We have
\begin{align*}
\theta [x_{i_1},[x_{i_2},\ldots,x_{i_m}]] &\equiv
 [x_{i_1}\inv,[x_{i_2},\ldots,x_{i_m}]^{(-1)^{m-1}}] \\
 &\equiv [x_{i_1},[x_{i_2},\ldots,x_{i_m}] ]^{ (-1)^m} (\operatorname{mod} N_{m+1}).
\end{align*}

(b) By (a).
\end{proof}

Suppose that $\f$ is an involution from $\aut N$ and
$\f_1,\f_2,\ldots,\f_m,\ldots$ are arbitrary
conjugates of $\f.$ For every
$\sigma$ in $\aut N$ let us construct the sequence
$$
\{\sigma_m(\f_1,\f_2,\ldots,\f_m) : m \in \N\}
$$
of automorphisms of $N$ as follows:
\begin{align*}
\s_0 &= \sigma,\\
\s_1 &= \f_1 \s_0 \f_1 \sigma_0\inv,\\
\s_2 &= \f_2 \s_1 \f_2 \sigma_1,\\
\s_3 &= \f_3 \s_2 \f_3 \sigma_2\inv,\\
&\ldots
\end{align*}
More formally, for every $m \ge 0$
\begin{equation}  \label{Sigma-Seq}
\sigma_{m+1} =
\begin{cases}
\f_{m+1} \sigma_m \f_{m+1} \sigma_m\inv, & \text{ if $m$ is even},\\
\f_{m+1} \sigma_m \f_{m+1} \sigma_m, & \text{ if $m$ is odd}.\\
\end{cases}
\end{equation}

The following result generalizes the corresponding
fact from \cite{ToCamb} proved there for free
nilpotent groups of nilpotency class $2.$

\begin{Prop}
Let $N$ be a free nilpotent group of nilpotency class
$s.$ Then an involution $\theta \in \aut N$ is a
symmetry modulo $\ia N$ {\em(}that is, coincides
with some symmetry modulo the group $\ia N${\em)}
if and only if for every
$\sigma$ from $\aut N$ and every tuple
$\theta_1,\theta_2,\ldots,\theta_s$ of conjugates of
$\theta$ the automorphism
$\sigma_s(\theta_1,\theta_2,\ldots,\theta_s)$ of $N$
is trivial.
\end{Prop}

\begin{proof} Suppose that $\theta= \theta^* \gamma,$
where $\theta^*$ is a symmetry and $\gamma$ is an
IA-automorphism. Since $\theta$ is an involution,
then $\theta^* \gamma = \gamma\inv \theta^*.$
Any member $\theta_k$ of the tuple
$\theta_1,\theta_2,\ldots,\theta_s,$ a symmetry
modulo $\ia N,$ also has the form $\theta^* \gamma_k$
for a suitable $\iat$-automorphism $\gamma_k.$

Let us prove by induction on $m$ that {\it the
automorphism $\sigma_m = \sigma_m(\theta_1,\ldots,\theta_m)$ is an
element of $K_m(N).$} This will follow the
necessity part of the Proposition.

Indeed, if $m=1,$ then $\sigma_m$ is an $\iat$-automorphism,
that is a member of $K_1(N).$ Assume that $\s_m \in K_m(N)$
and let $m$ be, for instance, even. We have by Lemma \ref{SymmBasics}(b)
and Lemma \ref{CommModK+1}:
\begin{align*}
\sigma_{m+1} &= \theta_{m+1} \sigma_m \theta_{m+1} \sigma_m\inv
             =\gamma_{m+1}\inv \theta^* \sigma_m \theta^* \gamma_{m+1} \sigma_m\inv \\
	     &\equiv \gamma_{m+1}\inv \sigma_m \gamma_{m+1} \sigma_m\inv
	     \equiv \id (\operatorname{mod} K_{m+1}(N)).
\end{align*}

Let us prove the converse. It is well-known that every
automorphism of the abelianization $\overline N$ of $N,$ the
free abelian group $N/[N,N],$ can
be lifted up to an automorphism of $N$ (see, for
instance, \cite[\S \ 4]{Malt} or \cite[Section 3.1, Section 4.2]{HN}).  Then it
suffices to prove that for every involution of
$\aut{\overline N},$ which is not $-\id,$ there
exist an infinite sequence of the form (\theequation),
constructed inside $\aut{\overline N},$ which contains
no trivial members.

It can be seen quite easily that every involution
$f \in \aut{\overline N}$, which is not $-\id,$ has two
$f$-invariant direct summands $B,C$ of $\overline N$
with $\overline N=B\oplus C$
and $\rank B=2;$ moreover, the action of $f$
on $B$ can chosen so that $f|_B$ is neither $\id_B,$
nor $-\id_B$ (\cite[Theorem 1.4]{ToCamb}, \cite[Lemma
1]{HuaRei}). This reduces the problem to the
automorphism groups of two-generator free abelian
groups; for the sake of simplicity we shall work with
the group $\gl{2,\Z}.$

According to the just mentioned result from \cite{HuaRei},
every involution in $\gl{2,\Z}$ is conjugate either
to the involution
$$
\begin{pmatrix}
1 & 0 \\
0 & -1
\end{pmatrix},
$$
or to the involution
$$
\begin{pmatrix}
0 & 1 \\
1 & 0
\end{pmatrix};
$$
hence the group $\gl{2,\Z}$ has exactly two conjugacy
classes of non-central involutions. One readily checks
that for every integer $m$
\begin{equation}
\begin{pmatrix}
1  & 0 \\
2m & -1
\end{pmatrix}
\sim
\begin{pmatrix}
1 & 0 \\
0 & -1
\end{pmatrix}
\text{ and }
\begin{pmatrix}
1 & 0 \\
2m-1 & -1
\end{pmatrix}
\sim
\begin{pmatrix}
0 & 1 \\
1 & 0
\end{pmatrix},
\end{equation}
where $\sim$ denotes the conjugacy relation.

Let $S$ be a non-central matrix
from $\gl{2,\Z}$ and $m$ an integer
number. Suppose that
\begin{align*}
X(m) &=
\begin{pmatrix}
1 & 0 \\
2m & -1
\end{pmatrix}
S
\begin{pmatrix}
1 & 0 \\
2m & -1
\end{pmatrix}
S,
\\
Y(m) &=
\begin{pmatrix}
1 & 0 \\
2m & -1
\end{pmatrix}
S
\begin{pmatrix}
1 & 0 \\
2m & -1
\end{pmatrix}
S\inv.
\end{align*}
There are no difficulties in the verification of the
following fact: some element of the `general' matrix
$X$ (and $Y$) depends linearly on $m.$ It follows
that for a suitable integer $m$ the matrix $X(m)$ ($Y(m)$)
is again non-central. This means that, starting with a
non-central matrix, we can construct an infinite sequence of the
form \eqref{Sigma-Seq} having no central matrices; in
particular, there will be no trivial matrices in this
sequence. Exactly the same argument, using matrices
of the form
$$
\begin{pmatrix}
1 & 0 \\
2m-1 & -1
\end{pmatrix},
$$
proves the similar result for the second conjugacy
class of non-central involutions in $\gl{2,\Z}.$
\end{proof}

\begin{Cor}
Symmetries modulo $\ia N$ form a definable
family in the group $\aut N.$
\end{Cor}

Let $T^-(N)$ denote the set of all automorphisms
$\{\s\}$ of $N$ such that for every $\theta,$ which is
a symmetry modulo $\ia N,$ the conjugate of
$\s$ by $\theta$ is equal to $\s\inv.$ Similarly
$T^+(N)$ denotes the set of all automorphisms
of $N,$ which commute with every symmetry modulo
$\ia N.$

\begin{Prop} \label{One-Step-Down}
Let $N$ be a free nilpotent group of nilpotency class
$s.$ Then $K_{s-1}(N) = T^+(N) \cup T^-(N).$ Therefore
$K_{s-1}(N),$ the kernel of a surjective homomorphism
from $\aut N$ to the automorphism group of a free
nilpotent group of nilpotency class $s-1$ and
of the same rank as one of $N,$ is a
definable subgroup of $\aut N.$
\end{Prop}

\begin{proof}
An arbitrary element $\s$ from $T^+(N) \cup T^-(N)$
must commute with any product of two symmetries: if,
for instance, $\s \in T^-(N),$ $\theta_1$ and
$\theta_2$ are two symmetries then
$$
\theta_1 \theta_2 \s (\theta_1 \theta_2)\inv =
\theta_1 \theta_2 \s \theta_2 \theta_1 =
\theta_1 \s\inv \theta_1 =\s.
$$
On the other hand, one finds among the automorphisms
of $N,$ which can be expressed as a product of
two symmetries, conjugations (inner
automorphisms of $N$) by primitive elements (that is,
members of bases of $N$).
This implies that $\s$ commutes with every
element of $\inn N.$ Hence $\s$ preserves
each element of $N$ modulo the center of $N.$
The center of $N$ is equal to the subgroup
$N_s$ \cite[Section 3.1]{HN}, and therefore $\s \in
K_{s-1}(N).$

Let $\tau$ be a conjugation by a primitive
element $x$ of $N.$ We are going to represent
$\tau$ as a product of two symmetries.
The element $x$ is a member of some basis $\cB$ of $N.$ Suppose
$\theta_1$ is a symmetry, which inverts
each element of $\cB.$ Then if a symmetry
$\theta_2$ is defined as follows
\begin{alignat*} 2
\theta_2 x &=x\inv,         &&\\
\theta_2 y &=x\inv y\inv x, &\quad  \forall y \in \cB \setminus \{x\}, &
\end{alignat*}
the product of $\theta_1 \theta_2$ of $\theta_1$ and
$\theta_2$ is equal to $\tau.$

Conversely, according to Lemma \ref{SymmBasics} (b)
every element of $K_{s-1}(N)$ either lies
in $T^+(N),$ or in $T^-(N).$
\end{proof}

\section{Interpretations}

\begin{Th} \label{Interp-of-Aut-of-2-step}
Let $N$ be a free nilpotent group of class $\ge 2.$
Then the automorphism group of $N$ first-order
interprets the automorphism group of a free nilpotent
group of class $2$ and of rank which is the same as
one of $N$ {\em(}uniformly in $N${\em).}
\end{Th}

\begin{proof}
By Proposition \ref{One-Step-Down}.
\end{proof}

Until otherwise stated, we shall assume that {\it $N$
is a free nilpotent group of class $2$} and
that {\it $A$ denotes the abelianization $\overline N$ of $N.$}

It can be shown that $\inn N,$ the subgroup of all
conjugations, is a $\varnothing$-definable subgroup of
$\aut N$ \cite[Corollary 3.2]{ToCamb}.  The group
$\inn N$ is isomorphic to the free abelian group
$A.$ Thus, we can interpret in $\aut N$ the free
abelian group $A$ and the automorphism group of $A$
with the action on the elements of $A.$

We can also interpret in $\aut N$ the family
$\cD$ of all direct summands of $A$ with inclusion
relation and a binary relation, say $R$ such
that
$$
R(B,C) \longleftrightarrow A = B \oplus C.
$$
One can prove that an involution $f$ from $\aut A$ is
diagonalizable in some basis of $A$ if only if there
are no elements of order three in the set $K(f)K(f),$
where $K(f)$ denotes the conjugacy class of $f$ (see
proof of Proposition 2.4 in \cite{ToCamb}). Hence the
fixed-point subgroups of diagonalizable involutions
can be used to interpret the direct summands. Having the
group $A$ interpreted in $\aut N,$ we can easily interpret the
inclusion relation and the relation $R$ on the
family $\cD.$

Summing up, we see that {\it the group $\aut N$ first-order
interprets the multi-sorted structure $\mathcal M$ with
the following description:}
\begin{itemize}
\item the sorts of $\mathcal M$ are the free abelian
group $A,$ its automorphism group $\aut A$
and the family $\cD$ of all direct summands of $A;$

\item all sorts carry their natural relations; the
relations of $\cD$ are the inclusion relation and the
relation $R;$

\item $\mathcal M$ has as one of the basic relations
the membership relation on $A \cup \cD;$

\item there are relations defining the action
of $\aut N$ on other sorts.
\end{itemize}

\begin{Lem} \label{cM-Interprets-Set-Theory}
Let $A$ be of infinite rank. Then the first-order
theory of the structure $\mathcal M$ syntactically
interprets the full second-order theory of
the set $|A|$ {\em(}in empty language{\em)},
uniformly in $A.$
\end{Lem}

\begin{proof}
It follows from the results in Section 4 of
\cite{ShCler}, that the first-order theory of the
endomorphism semi-group $\End A$ of $A$ syntactically
interprets $\operatorname{Th}_2(|A|)$ (for the sake of
convenience the reader may refer to \cite{ToAPAL},
where the very similar case of varieties of vector spaces
is considered in some details in the proof of Proposition 10.1;
an analysis of the proof shows that it
works also for free $\Z$-modules, or, in other words,
for free abelian groups).

To complete the proof, we could therefore interpret in
$\mathcal M$ the endomorphism semi-group of $A.$
There is a (folklore) trick by which the endomorphisms
can be interpreted in structures similar to $\mathcal M$
constructed over modules. This trick can be
briefly characterized as follows: three submodules,
such that any two of them are direct complements of
each other, are used to interpret the endomorphism
semi-group of one of them.  A detailed description of
the trick for infinite-dimensional vector spaces
can be found in \cite{ToAPAL} (see the
proof of Proposition 9.3); the reader is
again referred to \cite{ToAPAL} to see
that everything works for free $\Z$-modules
as well.
\end{proof}

The following result solves the problem posed
by S.~Shelah (see the Introduction) for all varieties
of nilpotent groups $\mathfrak N_s,$ where $s \ge 2.$

\begin{Th}
The first-order theory of the automorphism
group of any infinitely generated free nilpotent
group $N$ of class $\ge 2$ syntactically
interprets the full second-order theory
of the set $|N|.$ The first-order theory
of $\aut N$ is therefore unstable and
undecidable.
\end{Th}

\begin{proof}
By Theorem \ref{Interp-of-Aut-of-2-step}
and Lemma \ref{cM-Interprets-Set-Theory}.
\end{proof}

We have also solved the problem of classification of
elementary types of the automorphism groups of
infinitely generated free groups from varieties $\mathfrak
N_s:$

\begin{Th}
Let $N_1$ and $N_2$ be infinitely generated free
nilpotent groups of the same class $\ge 2.$ Then the
automorphism groups $\aut{N_1}$ and $\aut{N_2}$ are
elementarily equivalent if and only if the sets
$|N_1|$ and $|N_2|$ {\em(}with no structure{\em)} are
equivalent in the full second-order logic.
\end{Th}

\begin{proof}
By Theorem \ref{Interp-of-Aut-of-2-step}
and Lemma \ref{cM-Interprets-Set-Theory}.
\end{proof}

Let $N$ again denote a free nilpotent
group of class $2$ (recall that
$A$ stands for the abelianization
of $N$ and $\mathcal M$ is the multi-sorted
structure constructed over $A$).

We are going to estimate the logical
strength/complexity of the first-order theory of $\aut
N$ in the case, when $N$ is finitely generated.

\begin{Lem} \label{Z-in-cM}
Let $A$ be of rank at least $2.$ Then the structure $\mathcal M$
first-order interprets {\em(}with parameters{\em)} the ring of
integers $\Z.$
\end{Lem}

\begin{proof}
Let us consider two direct summands $B,C$ of
$A$ such that
$$
A = B\oplus C \text{ and } \rank B =2.
$$
Write $G$ for the group of all automorphisms of $A$
which preserve $B$ and point-wise fix $C.$ Clearly,
the structure $\str{G,B}$ with natural relations (that
is, with all relations on sorts along with the action
of $G$ on $B$) is isomorphic to the two-sorted
structure $\str{\gl{2,\Z},\Z^2}$ (taken in the same
language as one of $\str{G,B}$). It is a well-known
and simple result that the latter two-sorted structure
first-order interprets the ring of integers $\Z.$
\end{proof}

As an immediate corollary we have the following
fact.

\begin{Th} Suppose that $N$ is a finitely generated
free nilpotent group of class $\ge 2.$ Then the
first-order theory of the group $\aut N$ is unstable
and undecidable.
\end{Th}


\begin{thebibliography}{99}

\bibitem{Andr} S.~Andreadakis, {\it On the automorphisms
of free groups and free nilpotent groups}, {Proc.
London Math. Soc} {\bf 15} (1965), 239--268.



\bibitem{HuaRei} L.~K.~Hua, I.~Reiner, {\it Automorphisms of
the unimodular group}, {Trans. Amer. Math.
Soc.} {\bf 71} (1951), 331--348.


\bibitem{HN} H.~Neumann. {\it Varieties of Groups},
Springer-Verlag, 1967.

\bibitem{MKS} W.~Magnus, A.~Karrass,
D.~Solitar. {\it Combinatorial Group Theory}, Wiley, 1966.

\bibitem{Malt} A.~I.~Maltsev, {\it On algebras with identity
defining relations} (Russian), {Mat. Sb.}
{\bf 26} (1950), 19--33.

\bibitem{Sh1} S.~Shelah, {\it First-order theory of permutation
groups}, {Israel. J. Math.} {\bf 14} (1973), 149--162.


\bibitem{Sh2} S.~Shelah, {\it Errata to: first-order theory of
permutation groups}, {Israel J. Math.}  {\bf 15} (1973), 437--441.


\bibitem{ShCler} S.~Shelah, {\it Interpreting set
theory in the endomorphism semi-group of a free
algebra or in a category}, {Annales Scientifiques de L'universite
de Clermont} {\bf 13} (1976), 1--29.

\bibitem{ShMathJap} S.~Shelah, {\it On what I do not
understand {\em(}and have something to say{\em)},
model theory}, Math. Japon. {\bf 51} (2000), 329--377.


\bibitem{ToAPAL} V.~Tolstykh, {\it Elementary equivalence of
infinite-dimensional classical groups}, {Ann. Pure Appl.
Logic} {\bf 105} (2000), 103--156.


\bibitem{ToJLM1} V.~Tolstykh, {\it The automorphism
tower of a free group}, {J.  London Math.
Soc.} {\bf 61} (2000), 423--440.

\bibitem{ToJLM2} {\it Set theory is interpretable in the
automorphism group of an infinitely generated free
group}, {J. London Math. Soc.} {\bf 62} (2000),
16--26.

\bibitem{ToCamb} V.~Tolstykh, {\it Free two-step nilpotent groups
whose automorphism group is complete}, {Math. Proc.
Cambridge Philos. Soc.} {\bf 131} (2001), 73--90.

\end{thebibliography}
\end{document}